\documentclass[preprint,11pt]{elsarticle}

\usepackage[T1]{fontenc}
\usepackage[utf8]{inputenc}
\usepackage{lmodern}
\usepackage{amsmath,amssymb,amsthm,mathtools}
\usepackage{graphicx}
\usepackage[hidelinks]{hyperref}

\numberwithin{equation}{section}

\newtheorem{lemma}{Lemma}[section]
\newtheorem{theorem}{Theorem}[section]

\newtheorem{corollary}{Corollary}[section]

\newtheorem{remark}{Remark}[section]

\journal{}

\begin{document}

\begin{frontmatter}

\title{On the asymptotics of ground states for a boundary value problem
for the equation
\texorpdfstring{$-\varepsilon \Delta_p u
= a|u|^{q-2}u - b|u|^{\gamma-2}u$}{-epsilon Delta_p u = a|u|^{q-2}u - b|u|^{gamma-2}u}}

\author[inst1]{Yavdat Sh. Il'yasov\corref{cor1}}
\ead{ilyasov02@gmail.com}

\author[inst1]{Elvira I. Turianova}
\ead{elviratur820@gmail.com}

\cortext[cor1]{Corresponding author.}

\affiliation[inst1]{
	organization={Institute of Mathematics with Computing Centre,
	Ufa Federal Research Centre, Russian Academy of Sciences},
	addressline={Chernyshevsky str. 112},
	city={Ufa},
	postcode={450008},
	country={Russia}
}

\begin{abstract}
	We study a singularly perturbed Dirichlet problem for the \(p\)-Laplacian
	with competing superlinear terms,
	\[
	-\varepsilon \Delta_p u
	=
	a(x)|u|^{q-2}u
	-
	b(x)|u|^{\gamma-2}u,
	\qquad
	u|_{\partial\Omega}=0,
	\]
	where \(1<p<q<\gamma<p^*\), \(a\ge0\), \(b\ge\sigma_b>0\), and
	\(\varepsilon>0\) is small. By means of the nonlinear Rayleigh quotient
	method, we introduce two critical parameter values, \(\varepsilon^*\) and
	\(\varepsilon_e^*\), related respectively to the Nehari manifold and to the
	zero energy level. We prove nonexistence of nontrivial weak solutions for
	\(\varepsilon>\varepsilon^*\), and existence of at least two positive weak
	solutions for \(0<\varepsilon<\varepsilon_e^*\); one of them is a ground
	state.
	
	The main result describes the asymptotic behaviour of ground states as
	\(\varepsilon\to0^+\). If, in addition, \(a\ge\sigma_a>0\), then every family of positive ground states \(u_\varepsilon\) converges in measure to
	\[
	\bar u_0(x)=\left(\frac{a(x)}{b(x)}\right)^{1/(\gamma-q)}.
	\]
	The convergence is strong in \(L^r(\Omega)\) for \(1\le r<\gamma\) and weak
	in \(L^r(\Omega)\) for \(1<r\le\gamma\).
\end{abstract}


\end{frontmatter}

\section{Introduction}

In this paper we consider the following Dirichlet problem with a small
parameter \(\varepsilon>0\):
\begin{equation}\label{eq:1}
	\left\{
	\begin{aligned}
		-\varepsilon \Delta_p u
		&=
		a(x)|u|^{q-2}u
		-
		b(x)|u|^{\gamma-2}u,
		&& x\in\Omega,\\
		u&=0,
		&& x\in\partial\Omega .
	\end{aligned}
	\right.
\end{equation}
Here
\[
\Delta_p u
=
\operatorname{div}\!\left(|\nabla u|^{p-2}\nabla u\right)
\]
is the \(p\)-Laplacian, \(\Omega\subset\mathbb R^N\) is a bounded connected
domain with \(C^1\) boundary, and \(1<p<q<\gamma<p^*\), where
\[
p^* =
\begin{cases}
	\dfrac{pN}{N-p}, & \text{if } p<N, \\[4pt]
	+\infty,         & \text{if } p\ge N .
\end{cases}
\]
We assume that
\begin{equation}\label{eq:Condab}
	a,b\in L^\infty(\Omega),\qquad
	a(x)\ge0,\qquad
	b(x)\ge\sigma_b>0
	\quad\text{for a.e. }x\in\Omega ,
\end{equation}
where \(\sigma_b>0\) is a constant. Throughout the paper we also assume that
\(a\not\equiv0\).

A function \(u\in W^{1,p}_0(\Omega)\) is called a \emph{weak solution} of
problem~\eqref{eq:1} if it is a critical point of the energy functional
\begin{equation}\label{eq:energy}
	\Phi_\varepsilon(u)
	=
	\frac{\varepsilon}{p}\int_\Omega |\nabla u|^p\,dx
	-
	\frac1q\int_\Omega a|u|^q\,dx
	+
	\frac1\gamma\int_\Omega b|u|^\gamma\,dx .
\end{equation}
Here \(W^{1,p}_0:=W^{1,p}_0(\Omega)\) is the closure of
\(C^\infty_0(\Omega)\) in the Sobolev space \(W^{1,p}(\Omega)\).
A weak solution \(u\in W^{1,p}_0(\Omega)\) is called a \emph{ground state}
of \(\Phi_\varepsilon\) if
\[
\Phi_\varepsilon(u)\le \Phi_\varepsilon(w)
\]
for every weak solution \(w\in W^{1,p}_0(\Omega)\) of problem~\eqref{eq:1}.

Problem~\eqref{eq:1} is singularly perturbed: as \(\varepsilon\to0^+\), the
differential order disappears in the limit, and boundary layers may form near
\(\partial\Omega\). Already in the one-dimensional case \(p=2\), \(a=b=1\),
the problem becomes
\begin{equation}\label{eq:11}
	\left\{
	\begin{aligned}
		-\varepsilon u''(x)
		&=
		u^{q-1}(x)-u^{\gamma-1}(x),
		&& x\in(0,1),\\
		u(0)&=u(1)=0 .
	\end{aligned}
	\right.
\end{equation}
The formal method of matched asymptotic expansions
(see, for instance, \cite{Ilin1989,Nayfeh}) leads to the approximation
\[
u(x)
\sim
U\!\left(\frac{x}{\sqrt{\varepsilon}}\right)
+
U\!\left(\frac{1-x}{\sqrt{\varepsilon}}\right)
-
1,
\qquad 0\le x\le1,
\]
where the boundary-layer profile \(U\) is determined by
\[
-U''=U^{q-1}-U^{\gamma-1},
\qquad
U(0)=0,\qquad
U(+\infty)=1 .
\]
This profile admits the integral representation
\[
\int_0^{U(\xi)}
\frac{dt}{
	\sqrt{
		2\left(
		\frac{t^\gamma}{\gamma}
		-
		\frac{t^q}{q}
		+
		\frac1q
		-
		\frac1\gamma
		\right)
	}
}
=
\xi .
\]
The approximation is consistent with the boundary conditions in the leading
order: in the interior of the interval, both arguments
\(x/\sqrt{\varepsilon}\) and \((1-x)/\sqrt{\varepsilon}\) are large, so that
\(U\approx1\) and \(u(x)\approx1\), whereas near the endpoints the
corresponding boundary-layer profile provides the transition to the zero
boundary value.

Extending such a description to multidimensional domains, especially in the
presence of non-symmetric geometry or variable coefficients, is a substantial
mathematical difficulty. At the same time, singularly perturbed nonlinear
boundary value problems form a central topic in modern elliptic theory because
they combine delicate asymptotic analysis with a wide range of applications.
They arise, for instance, in population dynamics, models of infection spread,
control theory, mechanics of materials with memory, nonlinear optics,
physiological models, and predator--prey systems; see, for example,
\cite{Kadalbajoo,Murray,Roos,Volpert}.

A large part of the mathematical literature on singularly perturbed elliptic
problems is devoted to concentration phenomena. In this direction one usually
studies families of solutions which, as the perturbation parameter tends to
zero, develop sharply localized profiles: interior or boundary spikes,
spike-layer solutions, multi-peak solutions, or bubble-type solutions. The
main questions concern the existence of such concentrating solutions, the
number and location of their concentration points, the influence of the
geometry of \(\Omega\), and the role of the coefficients and lower-order terms
in selecting the limiting profiles. This circle of problems has been
extensively studied for semilinear and quasilinear elliptic equations and is
closely related to the analysis of concentration near critical points of
auxiliary potentials, mean-curvature type quantities, or other effective
variational landscapes; see, among many others,
\cite{Ambros_Malch,Ambr_Mal_Ni,Badiale,Dancer,Howes,Lin,Ni,Vasil,Jager}.

The problem considered in the present paper has a different emphasis. Rather
than constructing spike or bubble solutions and describing the fine geometry
of their concentration sets, we study the asymptotic behaviour of ground
states for a competing-superlinear Dirichlet problem with small diffusion.
In this setting the formal limiting equation is no longer an elliptic boundary
value problem, but a pointwise algebraic balance between the two nonlinear
terms. Our main result shows that, under the positivity assumption on \(a\),
ground states converge to the explicit profile
\[
\bar u_0(x)
=
\left(\frac{a(x)}{b(x)}\right)^{1/(\gamma-q)},
\]
with convergence in measure and strong convergence in \(L^r(\Omega)\) for
\(1\le r<\gamma\). Thus the singular limit is described not by localization at
isolated points, but by convergence to a spatially distributed equilibrium
profile determined by the variable coefficients.

A further difficulty is caused by the loss of differential order in the
singular limit. Formally setting \(\varepsilon=0\) reduces the elliptic
problem to the algebraic equation
\[
a(x)|u|^{q-2}u-b(x)|u|^{\gamma-2}u=0,
\]
so that the boundary condition is no longer encoded in the limiting equation.
Moreover, this equation does not select a unique limit. In the nonnegative
class, for example, every function of the form
\[
u(x)=\bar u_0(x)\chi_E(x),
\qquad
\bar u_0(x)=\left(\frac{a(x)}{b(x)}\right)^{1/(\gamma-q)},
\]
where \(E\subset\Omega\) is measurable, is a formal solution of the limiting
equation. Thus the limit problem has a continuum of possible solutions.

For this reason, the main asymptotic issue is not merely to identify the
algebraic balance between the nonlinearities, but to determine which of the
many possible limiting profiles is selected by ground states of the original
elliptic problem. The result proved below shows that, under the positivity
assumption on \(a\), the selected profile is the full positive branch
\(\bar u_0\). This also explains why the natural convergence statement is
formulated first in measure and then in \(L^r(\Omega)\), \(1\le r<\gamma\),
rather than as a classical convergence of solutions of elliptic boundary value
problems.

In the present paper, problem~\eqref{eq:1} is studied by means of the
nonlinear Rayleigh quotient method developed in
\cite{ilyaReil,Ilyasov2022}. Applying this method to the parameter
\(\varepsilon\) leads to the nonlinear generalized Rayleigh quotient
\[
\Upsilon(u)
:=
\frac{
	\left(\displaystyle\int_\Omega a|u|^q\,dx\right)^{
		\frac{\gamma-p}{\gamma-q}}
}{
	\displaystyle
	\int_\Omega |\nabla u|^p\,dx
	\left(\displaystyle\int_\Omega b|u|^\gamma\,dx\right)^{
		\frac{q-p}{\gamma-q}}
}.
\]
Since in the general setting we assume only \(a\ge0\), there may exist
nonzero functions \(u\in W^{1,p}_0(\Omega)\) such that
\[
\int_\Omega a|u|^q\,dx=0.
\]
Thus the extremal quantities are naturally defined on the set
\[
\mathcal D
:=
\left\{
u\in W^{1,p}_0(\Omega)\setminus\{0\}:
\int_\Omega a|u|^q\,dx>0
\right\}.
\]
In what follows, \(\Upsilon\) is considered on \(\mathcal D\).

Two extremal parameter values are associated with this quotient:
\[
\varepsilon^*
=
c_{p,q,\gamma}
\sup_{u\in\mathcal D}\Upsilon(u),
\qquad
\varepsilon_e^*
=
c_{e,p,q,\gamma}
\sup_{u\in\mathcal D}\Upsilon(u),
\]
where
\begin{equation}\label{eq:c-pq-intro}
	c_{p,q,\gamma}
	=
	\frac{\gamma-q}{\gamma-p}
	\left(
	\frac{q-p}{\gamma-p}
	\right)^{\frac{q-p}{\gamma-q}},
	\qquad
	c_{e,p,q,\gamma}
	=
	\frac{p(\gamma-q)}{q(\gamma-p)}
	\left(
	\frac{\gamma(q-p)}{q(\gamma-p)}
	\right)^{\frac{q-p}{\gamma-q}} .
\end{equation}
Moreover,
\[
0<\varepsilon_e^*<\varepsilon^*<+\infty .
\]

The first main result describes existence and nonexistence of positive
solutions of problem~\eqref{eq:1} depending on the parameter
\(\varepsilon\).

\begin{theorem}\label{thm1}
	Let \(\Omega\subset\mathbb R^N\) be a bounded connected domain with
	\(C^1\) boundary, let \(1<p<q<\gamma<p^*\), and assume that
	\eqref{eq:Condab} holds. Then the following assertions hold.
	
	\begin{enumerate}
		\item[\textbf{(1)}]
		If \(\varepsilon>\varepsilon^*\), then problem~\eqref{eq:1} has no
		nontrivial weak solutions in \(W^{1,p}_0(\Omega)\).
		
		\item[\textbf{(2)}]
		For every \(0<\varepsilon<\varepsilon_e^*\), problem~\eqref{eq:1}
		has a positive weak solution
		\[
		u_\varepsilon\in W^{1,p}_0(\Omega).
		\]
		Moreover,
		\begin{enumerate}
			\item[(i)]
			\(u_\varepsilon\) is a ground state;
			
			\item[(ii)]
			\[
			\Phi_\varepsilon(u_\varepsilon)<0,
			\qquad
			\left.
			\frac{d^2}{dt^2}\Phi_\varepsilon(tu_\varepsilon)
			\right|_{t=1}
			>0;
			\]
			
			\item[(iii)]
			\(u_\varepsilon\in C^{1,\kappa}_{\mathrm{loc}}(\Omega)\)
			for some \(\kappa\in(0,1)\).
		\end{enumerate}
		
		\item[\textbf{(3)}]
		For every \(0<\varepsilon<\varepsilon_e^*\), problem~\eqref{eq:1}
		has a second positive weak solution
		\[
		v_\varepsilon\in W^{1,p}_0(\Omega).
		\]
		In addition,
		\[
		\Phi_\varepsilon(v_\varepsilon)>0,
		\qquad
		v_\varepsilon\in C^{1,\kappa}_{\mathrm{loc}}(\Omega)
		\]
		for some \(\kappa\in(0,1)\).
	\end{enumerate}
\end{theorem}

\begin{remark}
	In Theorem~\ref{thm1}, regularity is asserted only locally in the
	interior of the domain. This is because, for a \(C^1\) boundary, one
	cannot in general expect global \(C^{1,\kappa}\)-regularity of solutions
	up to \(\partial\Omega\).
\end{remark}

\begin{corollary}[The classical case \(p=2\)]
	Let \(p=2\), let \(\partial\Omega\) be of class \(C^{2,\alpha}\), and
	let \(a,b\in C^{0,\alpha}(\overline\Omega)\), with
	\eqref{eq:Condab} satisfied. Then the positive weak solutions
	\(u_\varepsilon\) and \(v_\varepsilon\) obtained in
	Theorem~\ref{thm1} are classical solutions of problem~\eqref{eq:1}.
	In particular,
	\[
	u_\varepsilon,\ v_\varepsilon\in C^2(\Omega)\cap C(\overline\Omega).
	\]
\end{corollary}

The next result describes the asymptotic behaviour of ground states as
\(\varepsilon\to0^+\).

\begin{theorem}\label{thmCC2}
	Let \(\Omega\subset\mathbb R^N\) be a bounded connected domain with
	\(C^1\) boundary, let \(1<p<q<\gamma<p^*\), assume that
	\eqref{eq:Condab} holds, and suppose that there exists a constant
	\(\sigma_a>0\) such that
	\[
	a(x)\ge\sigma_a
	\quad\text{for a.e. }x\in\Omega .
	\]
	Let \(u_\varepsilon\) be any family of positive ground states of
	problem~\eqref{eq:1}, \(0<\varepsilon<\varepsilon_e^*\). Then, as
	\(\varepsilon\to0^+\),
	\[
	u_\varepsilon\to\bar u_0
	\quad\text{in measure in }\Omega,
	\]
	where
	\begin{equation}\label{eq:limit-profile}
		\bar u_0(x)
		=
		\left(\frac{a(x)}{b(x)}\right)^{1/(\gamma-q)}
		\quad\text{for a.e. }x\in\Omega .
	\end{equation}
	Moreover,
	\begin{equation}\label{eq:strong-Lr}
		u_\varepsilon-\bar u_0\to0
		\quad\text{strongly in }L^r(\Omega),
		\qquad
		1\le r<\gamma,
	\end{equation}
	and
	\[
	u_\varepsilon-\bar u_0\rightharpoonup0
	\quad\text{weakly in }L^r(\Omega),
	\qquad
	1<r\le\gamma .
	\]
	The limit function \(\bar u_0\) satisfies the limiting equation
	\[
	a(x)\bar u_0^{q-1}(x)
	-
	b(x)\bar u_0^{\gamma-1}(x)
	=
	0
	\quad\text{for a.e. }x\in\Omega .
	\]
\end{theorem}

\begin{figure}[!ht]
	\begin{minipage}[h]{0.9\linewidth}
		\center{\includegraphics[scale=1.3]{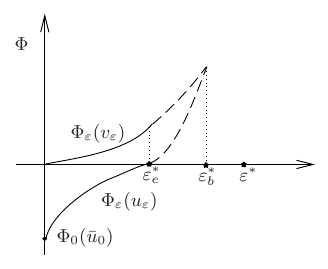}}
		\caption{A schematic picture of the energy levels of two branches of positive
			solutions up to the threshold \(\varepsilon_e^*\). Dashed lines indicate the
			expected continuation of the branches.}
		\label{fig1}
	\end{minipage}
\end{figure}

By the simple scaling \(u\mapsto tu\), problem~\eqref{eq:1} with
\(\varepsilon>0\) can be transformed into either of the following equivalent
forms:
\begin{equation}\label{eq:1lam}
	\left\{
	\begin{aligned}
		-\Delta_p u
		&=
		\lambda a(x)|u|^{q-2}u
		-
		b(x)|u|^{\gamma-2}u,
		&& x\in\Omega,\\
		u&=0,
		&& x\in\partial\Omega,
	\end{aligned}
	\right.
\end{equation}
and
\begin{equation}\label{eq:1nu}
	\left\{
	\begin{aligned}
		-\Delta_p u
		&=
		a(x)|u|^{q-2}u
		-
		\nu b(x)|u|^{\gamma-2}u,
		&& x\in\Omega,\\
		u&=0,
		&& x\in\partial\Omega,
	\end{aligned}
	\right.
\end{equation}
where \(\lambda,\nu>0\).

These problems are in one-to-one correspondence with problem~\eqref{eq:1}
whenever \(\varepsilon\lambda\nu\neq0\). Namely, if
\(u_\varepsilon\in W^{1,p}_0(\Omega)\) is a weak solution of
\eqref{eq:1}, then the functions
\begin{align}
	u_\lambda
	&:=
	\varepsilon^{-1/(\gamma-p)}u_\varepsilon,
	&
	\lambda
	&=
	\varepsilon^{-(\gamma-q)/(\gamma-p)},
	\label{eq:scaling-lambda}
	\\
	u_\nu
	&:=
	\varepsilon^{-1/(q-p)}u_\varepsilon,
	&
	\nu
	&=
	\varepsilon^{(\gamma-q)/(q-p)}
	\label{eq:scaling-nu}
\end{align}
are weak solutions of problems~\eqref{eq:1lam} and~\eqref{eq:1nu},
respectively. The inverse transformations are
\[
u_\varepsilon
=
\lambda^{-1/(\gamma-q)}u_\lambda,
\qquad
u_\varepsilon
=
\nu^{1/(\gamma-q)}u_\nu .
\]
Hence Theorems~\ref{thm1} and~\ref{thmCC2} imply the following assertion.

\begin{corollary}\label{cor:scaled-asymptotics}
	Assume that the hypotheses of Theorem~\ref{thm1} are satisfied. Then:
	\begin{enumerate}
		\item[(1)]
		Problems~\eqref{eq:1lam} and~\eqref{eq:1nu} possess ground states
		\(u_\lambda\) and \(u_\nu\) for
		\[
		\lambda\in(\lambda_e^*,+\infty),
		\qquad
		\nu\in(0,\nu_e^*),
		\]
		respectively, where
		\[
		\lambda_e^*
		=
		(\varepsilon_e^*)^{-(\gamma-q)/(\gamma-p)},
		\qquad
		\nu_e^*
		=
		(\varepsilon_e^*)^{(\gamma-q)/(q-p)}.
		\]
		
		\item[(2)]
		If, in addition, the hypotheses of Theorem~\ref{thmCC2} are
		satisfied, then, as \(\lambda\to+\infty\),
		\[
		\lambda^{-1/(\gamma-q)}u_\lambda
		\to
		\bar u_0
		\quad\text{in measure in }\Omega,
		\]
		\[
		\lambda^{-1/(\gamma-q)}u_\lambda-\bar u_0
		\to0
		\quad\text{strongly in }L^r(\Omega),
		\qquad
		1\le r<\gamma,
		\]
		and
		\[
		\lambda^{-1/(\gamma-q)}u_\lambda-\bar u_0
		\rightharpoonup0
		\quad\text{weakly in }L^\gamma(\Omega).
		\]
		
		Similarly, as \(\nu\to0^+\),
		\[
		\nu^{1/(\gamma-q)}u_\nu
		\to
		\bar u_0
		\quad\text{in measure in }\Omega,
		\]
		\[
		\nu^{1/(\gamma-q)}u_\nu-\bar u_0
		\to0
		\quad\text{strongly in }L^r(\Omega),
		\qquad
		1\le r<\gamma,
		\]
		and
		\[
		\nu^{1/(\gamma-q)}u_\nu-\bar u_0
		\rightharpoonup0
		\quad\text{weakly in }L^\gamma(\Omega).
		\]
		
		In particular, in both cases weak convergence holds in
		\(L^r(\Omega)\) for all \(1<r\le\gamma\).
	\end{enumerate}
\end{corollary}

\begin{remark}
	Problems of the form \eqref{eq:1}, \eqref{eq:1lam}, and \eqref{eq:1nu}
	belong to the class of quasilinear parametric problems in which extremal
	parameter values and their variational characterization through nonlinear
	generalized Rayleigh quotients play a central role; see, for instance,
	\cite{ilBob,Il2005,Ilyas2017,IlyasKaye,RamosQuoirinSicilianoSilva}.
	
	In terms of their general variational structure, these problems are close
	to well-known parametric equations with competing nonlinearities, including
	concave--convex type problems that have been actively studied in recent
	decades; see, for example,
	\cite{AmbrosettiBrezisCerami,BartschWillem,Il2005}. However, the case
	considered here has essential qualitative differences. When
	\[
	1<p<q<\gamma,
	\]
	both nonlinearities are superlinear relative to the \(p\)-Laplacian, and
	the competition between the positive term of order \(q\) and the negative
	term of order \(\gamma\) leads to a different geometry of the energy
	functional and to a different structure of extremal parameters and
	bifurcation diagrams.
\end{remark}


\section{Preliminaries}
\label{sec:3}

In what follows, \(L^r:=L^r(\Omega)\), \(1<r<+\infty\), denotes the space of
measurable functions with finite norm
\[
\|u\|_r
=
\left(\int_\Omega |u|^r\,dx\right)^{1/r}.
\]
By the Poincaré--Friedrichs inequality, the norm in \(W^{1,p}_0(\Omega)\) can
be equivalently defined by
\[
\|u\|_{1,p}
=
\left(\int_\Omega |\nabla u|^p\,dx\right)^{1/p}.
\]

We set
\[
T(u):=\int_\Omega |\nabla u|^p\,dx,
\qquad
A(u):=\int_\Omega a(x)|u|^q\,dx,
\qquad
B(u):=\int_\Omega b(x)|u|^\gamma\,dx .
\]
Since \(b(x)\ge\sigma_b>0\) for a.e. \(x\in\Omega\), we have
\[
B(u)>0
\quad
\text{for all }u\in W^{1,p}_0(\Omega)\setminus\{0\}.
\]
We also assume throughout that \(a\not\equiv0\). Hence the set
\[
\mathcal D
:=
\left\{
u\in W^{1,p}_0(\Omega)\setminus\{0\}: A(u)>0
\right\}
\]
is nonempty.

The Nehari manifold associated with problem~\eqref{eq:1} is given by
\[
\mathcal N_\varepsilon
:=
\left\{
u\in W^{1,p}_0(\Omega)\setminus\{0\}:
\varepsilon T(u)-A(u)+B(u)=0
\right\}.
\]
Equivalently,
\[
D_u\Phi_\varepsilon(u)(u)=0 .
\]
Observe that, for \(\varepsilon>0\), every function
\(u\in\mathcal N_\varepsilon\) satisfies
\[
A(u)=\varepsilon T(u)+B(u)>0.
\]
Therefore, when positive values of the parameter are considered, it is natural
to restrict attention to directions \(u\in\mathcal D\). Only such directions
can generate an intersection of the ray \(\{su:s>0\}\) with the Nehari
manifold for \(\varepsilon>0\).

Following the approach based on nonlinear generalized Rayleigh quotients
\cite{Ilyasov2022}, we introduce two Rayleigh quotients associated,
respectively, with the Nehari manifold and with the zero energy level. The
first quotient is defined by
\[
R_{\mathcal N}(u)
:=
\frac{A(u)-B(u)}{T(u)},
\qquad
u\in W^{1,p}_0(\Omega)\setminus\{0\}.
\]
Then
\[
\mathcal N_\varepsilon
=
\left\{
u\in W^{1,p}_0(\Omega)\setminus\{0\}:
R_{\mathcal N}(u)=\varepsilon
\right\}.
\]

Fix \(u\in\mathcal D\) and consider the function
\[
s\mapsto R_{\mathcal N}(su),
\qquad s>0.
\]
We have
\[
R_{\mathcal N}(su)
=
\frac{s^qA(u)-s^\gamma B(u)}{s^pT(u)}
=
\frac{A(u)s^{q-p}-B(u)s^{\gamma-p}}{T(u)} .
\]
Since \(1<p<q<\gamma\), this function has a unique positive critical point,
which is the point of its global maximum. A direct computation gives
\begin{equation}
	s_{\mathcal N}(u)
	=
	\left(
	\frac{(q-p)A(u)}
	{(\gamma-p)B(u)}
	\right)^{1/(\gamma-q)} .
	\label{eq:smu}
\end{equation}

Substituting \(s=s_{\mathcal N}(u)\) into \(R_{\mathcal N}(su)\), we obtain
the nonlinear generalized Rayleigh quotient
\begin{equation}
	\varepsilon(u)
	:=
	R_{\mathcal N}(s_{\mathcal N}(u)u)
	=
	c_{p,q,\gamma}
	\frac{
		A(u)^{\frac{\gamma-p}{\gamma-q}}
	}{
		T(u)B(u)^{\frac{q-p}{\gamma-q}}
	},
	\qquad u\in\mathcal D,
	\label{eq:mu_u}
\end{equation}
where
\begin{equation}
	c_{p,q,\gamma}
	=
	\frac{\gamma-q}{\gamma-p}
	\left(
	\frac{q-p}{\gamma-p}
	\right)^{\frac{q-p}{\gamma-q}} .
	\label{eq:c_pq}
\end{equation}

We define the extremal value
\begin{equation}
	\varepsilon^*
	:=
	\sup_{u\in\mathcal D}\varepsilon(u).
	\label{eq:mustar}
\end{equation}
Since \(a\not\equiv0\), the set \(\mathcal D\) is nonempty; hence
\[
\varepsilon^*>0.
\]
We now show that \(\varepsilon^*<+\infty\). By the Hölder,
Poincaré--Friedrichs, and Sobolev inequalities, there exists a constant
\(C>0\), independent of \(u\), such that
\[
A(u)
=
\int_\Omega a(x)|u|^q\,dx
\le
C
T(u)^{\frac{\gamma-q}{\gamma-p}}
B(u)^{\frac{q-p}{\gamma-p}} .
\]
Indeed, since \(q\in(p,\gamma)\), the \(L^q\)-norm is interpolated between
the \(L^p\)- and \(L^\gamma\)-norms, while the \(L^p\)-norm is controlled by
\(\|\nabla u\|_p\). The condition \(b(x)\ge\sigma_b>0\) allows us to replace
the \(L^\gamma\)-norm by the corresponding \(B(u)\)-term. Substituting this
estimate into \eqref{eq:mu_u}, we obtain
\begin{equation}
	0<\varepsilon^*<+\infty .
	\label{eq:214}
\end{equation}

The second Rayleigh quotient corresponds to the zero energy level
(cf. \cite{Carles,MarcCarlIl,ilyaReil,Ilyasov2022}):
\[
R_e(u)
:=
\frac{
	\frac1qA(u)-\frac1\gamma B(u)
}{
	\frac1pT(u)
},
\qquad
u\in W^{1,p}_0(\Omega)\setminus\{0\}.
\]
Then
\[
R_e(u)=\varepsilon
\quad
\Longleftrightarrow
\quad
\Phi_\varepsilon(u)=0 .
\]

For fixed \(u\in\mathcal D\), the function
\[
s\mapsto R_e(su),
\qquad s>0,
\]
has a unique positive critical point. It is given by
\begin{equation}
	s_e(u)
	=
	\left(
	\frac{\gamma(q-p)A(u)}
	{q(\gamma-p)B(u)}
	\right)^{1/(\gamma-q)} .
	\label{eq:smu_e}
\end{equation}
Substituting \(s=s_e(u)\) into \(R_e(su)\), we obtain
\begin{equation}
	\varepsilon_e(u)
	:=
	R_e(s_e(u)u)
	=
	c_{e,p,q,\gamma}
	\frac{
		A(u)^{\frac{\gamma-p}{\gamma-q}}
	}{
		T(u)B(u)^{\frac{q-p}{\gamma-q}}
	},
	\qquad u\in\mathcal D,
	\label{eq:mu_e_u}
\end{equation}
where
\begin{equation}
	c_{e,p,q,\gamma}
	=
	\frac{p(\gamma-q)}{q(\gamma-p)}
	\left(
	\frac{\gamma(q-p)}{q(\gamma-p)}
	\right)^{\frac{q-p}{\gamma-q}} .
	\label{eq:c_pq_e}
\end{equation}
Set
\[
\varepsilon_e^*
:=
\sup_{u\in\mathcal D}\varepsilon_e(u).
\]
The same estimate as above yields
\[
0<\varepsilon_e^*<+\infty .
\]

\begin{lemma}
	\label{lem:inters}
	For every \(u\in\mathcal D\) and every \(s>0\), the following equivalence
	holds:
	\[
	R_e(su)=R_{\mathcal N}(su)
	\quad
	\Longleftrightarrow
	\quad
	s=s_e(u).
	\]
\end{lemma}

\begin{proof}
	Fix \(u\in\mathcal D\) and put
	\[
	\psi(s):=R_e(su),
	\qquad s>0.
	\]
	Then
	\[
	\psi(s)
	=
	\frac{p}{T(u)}
	\left(
	\frac{A(u)}q s^{q-p}
	-
	\frac{B(u)}\gamma s^{\gamma-p}
	\right),
	\]
	and
	\[
	\psi'(s)
	=
	\frac{p}{T(u)}
	\left(
	\frac{q-p}{q}A(u)s^{q-p-1}
	-
	\frac{\gamma-p}{\gamma}B(u)s^{\gamma-p-1}
	\right).
	\]
	On the other hand,
	\[
	R_{\mathcal N}(su)-R_e(su)
	=
	\frac{1}{T(u)}
	\left(
	\frac{q-p}{q}A(u)s^{q-p}
	-
	\frac{\gamma-p}{\gamma}B(u)s^{\gamma-p}
	\right).
	\]
	Therefore
	\[
	\psi'(s)
	=
	\frac{p}{s}
	\left(
	R_{\mathcal N}(su)-R_e(su)
	\right).
	\]
	Hence
	\[
	R_e(su)=R_{\mathcal N}(su)
	\quad
	\Longleftrightarrow
	\quad
	\psi'(s)=0 .
	\]
	Since the equation \(\psi'(s)=0\) has a unique positive solution, namely
	\(s_e(u)\), the lemma follows.
\end{proof}

\begin{lemma}
	\label{lem:eps-greater-eps-e}
	For every \(u\in\mathcal D\), the strict inequality
	\[
	\varepsilon(u)>\varepsilon_e(u)
	\]
	holds. Moreover,
	\[
	\varepsilon^*>\varepsilon_e^* .
	\]
\end{lemma}

\begin{proof}
	It follows from formulas \eqref{eq:mu_u} and \eqref{eq:mu_e_u} that
	\(\varepsilon(u)\) and \(\varepsilon_e(u)\) differ only by constant
	factors. Thus it is enough to prove that
	\[
	c_{p,q,\gamma}
	>
	c_{e,p,q,\gamma}.
	\]
	We have
	\[
	\frac{c_{p,q,\gamma}}
	{c_{e,p,q,\gamma}}
	=
	\frac{q}{p}
	\left(\frac{q}{\gamma}\right)^{
		\frac{q-p}{\gamma-q}
	}.
	\]
	Let us show that this quantity is greater than \(1\). Denote
	\[
	\theta:=\frac{p}{q}\in(0,1),
	\qquad
	x:=\frac{\gamma}{q}>1.
	\]
	Then the required inequality is equivalent to
	\[
	x^{\frac{1-\theta}{x-1}}<\frac1\theta .
	\]
	After taking logarithms, it is enough to verify that
	\[
	\frac{1-\theta}{x-1}\ln x<-\ln\theta .
	\]
	Since \(x>1\), we have \(\ln x<x-1\), and since \(0<\theta<1\), we have
	\(1-\theta<-\ln\theta\). Consequently,
	\[
	\frac{1-\theta}{x-1}\ln x
	<
	1-\theta
	<
	-\ln\theta .
	\]
	Thus
	\[
	c_{p,q,\gamma}
	>
	c_{e,p,q,\gamma}.
	\]
	Hence
	\[
	\varepsilon(u)>\varepsilon_e(u)
	\quad
	\text{for all }u\in\mathcal D.
	\]
	
	Since
	\[
	\varepsilon^*
	=
	c_{p,q,\gamma}
	\sup_{u\in\mathcal D}
	\frac{
		A(u)^{\frac{\gamma-p}{\gamma-q}}
	}{
		T(u)B(u)^{\frac{q-p}{\gamma-q}}
	},
	\]
	whereas
	\[
	\varepsilon_e^*
	=
	c_{e,p,q,\gamma}
	\sup_{u\in\mathcal D}
	\frac{
		A(u)^{\frac{\gamma-p}{\gamma-q}}
	}{
		T(u)B(u)^{\frac{q-p}{\gamma-q}}
	},
	\]
	and this common supremum is positive and finite, we obtain
	\[
	\varepsilon^*>\varepsilon_e^* .
	\]
	The lemma is proved.
\end{proof}


\section{Proof of Theorem~\ref{thm1}}

We first prove \textbf{(1)}. Suppose, arguing by contradiction, that for
some \(\varepsilon>\varepsilon^*\) problem~\eqref{eq:1} has a nontrivial weak
solution \(\bar u\in W^{1,p}_0(\Omega)\), \(\bar u\not\equiv0\). Then
\[
D_u\Phi_\varepsilon(\bar u)(\bar u)=0,
\]
that is,
\[
\varepsilon\int_\Omega |\nabla \bar u|^p\,dx
-
\int_\Omega a(x)|\bar u|^q\,dx
+
\int_\Omega b(x)|\bar u|^\gamma\,dx
=0.
\]
Since \(\varepsilon>0\), \(\bar u\not\equiv0\), and \(b\ge\sigma_b>0\), it
follows that
\[
\int_\Omega a(x)|\bar u|^q\,dx>0.
\]
Hence \(\bar u\in\mathcal D\). Moreover,
\[
R_{\mathcal N}(\bar u)=\varepsilon .
\]
By the definition of \(\varepsilon(\bar u)\) as the maximum of
\(s\mapsto R_{\mathcal N}(s\bar u)\), we have
\[
\varepsilon
=
R_{\mathcal N}(\bar u)
\le
\varepsilon(\bar u)
\le
\varepsilon^*
<
\varepsilon ,
\]
which is impossible. Thus, for \(\varepsilon>\varepsilon^*\),
problem~\eqref{eq:1} has no nontrivial weak solutions.

\medskip

We now prove \textbf{(2)}.

\begin{lemma}\label{lem:coer}
	For every \(\varepsilon>0\), the functional \(\Phi_\varepsilon\) is
	coercive on \(W^{1,p}_0(\Omega)\).
\end{lemma}

\begin{proof}
	Since \(a,b\in L^\infty(\Omega)\), \(b\ge\sigma_b>0\), and \(q<\gamma\),
	there exist constants \(C_1,C_2>0\), independent of \(u\), such that
	\[
	\int_\Omega a(x)|u|^q\,dx
	\le
	C_1\|u\|_\gamma^q,
	\qquad
	\int_\Omega b(x)|u|^\gamma\,dx
	\ge
	C_2\|u\|_\gamma^\gamma .
	\]
	Therefore
	\[
	\Phi_\varepsilon(u)
	\ge
	\frac{\varepsilon}{p}\|\nabla u\|_p^p
	-
	C_1\|u\|_\gamma^q
	+
	\frac{C_2}{\gamma}\|u\|_\gamma^\gamma .
	\]
	Let \(\|u_n\|_{W^{1,p}_0}\to+\infty\). If \(\|u_n\|_\gamma\) is bounded,
	then the first term on the right-hand side tends to \(+\infty\). If
	\(\|u_n\|_\gamma\to+\infty\), then, since \(q<\gamma\), the positive term
	of order \(\|u_n\|_\gamma^\gamma\) dominates the negative term of order
	\(\|u_n\|_\gamma^q\). In both cases,
	\[
	\Phi_\varepsilon(u_n)\to+\infty .
	\]
	The lemma is proved.
\end{proof}

Let \(0<\varepsilon<\varepsilon_e^*\). Consider
\[
\hat\Phi_\varepsilon
:=
\inf_{u\in W^{1,p}_0(\Omega)}\Phi_\varepsilon(u).
\]
By the definition of \(\varepsilon_e^*\), there exists \(u_0\in\mathcal D\)
such that
\[
\varepsilon<\varepsilon_e(u_0).
\]
Set
\[
w_0:=s_e(u_0)u_0 .
\]
Then, by the definition of \(\varepsilon_e(u_0)\),
\[
R_e(w_0)=\varepsilon_e(u_0)>\varepsilon .
\]
Since
\[
R_e(w_0)
=
\frac{
	\frac1q\int_\Omega a(x)|w_0|^q\,dx
	-
	\frac1\gamma\int_\Omega b(x)|w_0|^\gamma\,dx
}{
	\frac1p\int_\Omega |\nabla w_0|^p\,dx
},
\]
the inequality \(R_e(w_0)>\varepsilon\) is equivalent to
\[
\Phi_\varepsilon(w_0)<0 .
\]
Consequently,
\begin{equation}\label{eq:negative-min-level}
	\hat\Phi_\varepsilon<0 .
\end{equation}

Let \((u_m)\subset W^{1,p}_0(\Omega)\) be a minimizing sequence:
\[
\Phi_\varepsilon(u_m)\to\hat\Phi_\varepsilon .
\]
By Lemma~\ref{lem:coer}, the sequence \((u_m)\) is bounded in
\(W^{1,p}_0(\Omega)\). Hence, passing to a subsequence, we may assume that
there exists \(u_\varepsilon\in W^{1,p}_0(\Omega)\) such that
\[
u_m\rightharpoonup u_\varepsilon
\quad\text{weakly in }W^{1,p}_0(\Omega),
\]
and, by the compact Sobolev embedding,
\[
u_m\to u_\varepsilon
\quad\text{strongly in }L^r(\Omega),
\qquad
1<r<p^* .
\]
In particular,
\[
u_m\to u_\varepsilon
\quad\text{strongly in }L^q(\Omega)
\quad\text{and}\quad
u_m\to u_\varepsilon
\quad\text{strongly in }L^\gamma(\Omega).
\]
Using the weak lower semicontinuity of the norm \(\|\nabla u\|_p\) and the
strong convergence of the nonlinear terms, we obtain
\[
\Phi_\varepsilon(u_\varepsilon)
\le
\liminf_{m\to +\infty}\Phi_\varepsilon(u_m)
=
\hat\Phi_\varepsilon .
\]
Therefore
\[
\Phi_\varepsilon(u_\varepsilon)=\hat\Phi_\varepsilon .
\]
In particular, by \eqref{eq:negative-min-level},
\[
\Phi_\varepsilon(u_\varepsilon)<0=\Phi_\varepsilon(0),
\]
and hence \(u_\varepsilon\not\equiv0\).

Since
\[
\Phi_\varepsilon(|u|)=\Phi_\varepsilon(u),
\qquad
|u|\in W^{1,p}_0(\Omega),
\]
we may replace \(u_\varepsilon\) by \(|u_\varepsilon|\) and assume that
\[
u_\varepsilon\ge0
\quad\text{in }\Omega .
\]
Since \(u_\varepsilon\) is a global minimizer of \(\Phi_\varepsilon\) on
\(W^{1,p}_0(\Omega)\), we have
\[
D_u\Phi_\varepsilon(u_\varepsilon)=0 .
\]
Thus \(u_\varepsilon\) is a weak solution of problem~\eqref{eq:1}.

Let us show that this solution is positive. Standard a priori estimates for
weak solutions of quasilinear elliptic equations with subcritical growth imply
that
\[
u_\varepsilon\in L^\infty(\Omega).
\]

Let us show that this solution is positive. We first recall that
\(u_\varepsilon\) is bounded. Indeed, since \(a,b\in L^\infty(\Omega)\) and
\[
1<p<q<\gamma<p^*,
\]
the right-hand side of equation~\eqref{eq:1} has subcritical growth in
\(u_\varepsilon\). Therefore the standard \(L^\infty\)-estimates for weak
solutions of quasilinear elliptic equations, obtained for instance by the
Moser iteration scheme, imply that
\[
u_\varepsilon\in L^\infty(\Omega);
\]
see, e.g., \cite{DrabekKufnerNicolosi1997}. Hence
\[
a(x)u_\varepsilon^{q-1}
-
b(x)u_\varepsilon^{\gamma-1}
\in L^\infty(\Omega).
\]
By the interior regularity theory for degenerate quasilinear elliptic
equations \cite{DiBened,Tolksdorf}, we obtain
\[
u_\varepsilon\in C^{1,\kappa}_{\mathrm{loc}}(\Omega)
\]
for some \(\kappa\in(0,1)\).

We now prove that \(u_\varepsilon>0\) in \(\Omega\). Since
\(u_\varepsilon\ge0\), \(u_\varepsilon\not\equiv0\), and
\(u_\varepsilon\in L^\infty(\Omega)\), there exists \(M>0\) such that
\[
0\le u_\varepsilon(x)\le M
\quad\text{for a.e. }x\in\Omega .
\]
Since \(\gamma>p\), we have
\[
u_\varepsilon^{\gamma-1}
=
u_\varepsilon^{\gamma-p}u_\varepsilon^{p-1}
\le
M^{\gamma-p}u_\varepsilon^{p-1}.
\]
Consequently,
\[
a(x)u_\varepsilon^{q-1}
-
b(x)u_\varepsilon^{\gamma-1}
\ge
-\|b\|_\infty M^{\gamma-p}u_\varepsilon^{p-1}
\quad\text{for a.e. }x\in\Omega .
\]
It follows from the equation for \(u_\varepsilon\) that, with some constant
\(C>0\),
\[
-\varepsilon\Delta_p u_\varepsilon
+
C u_\varepsilon^{p-1}
\ge0
\quad\text{in the weak sense in }\Omega .
\]
Equivalently,
\[
\Delta_p u_\varepsilon
-
\frac{C}{\varepsilon}u_\varepsilon^{p-1}
\le0
\quad\text{in the weak sense in }\Omega .
\]
By the strong maximum principle for the \(p\)-Laplacian
(see, for instance, \cite{Vazquez1984}) and by the connectedness of
\(\Omega\), we conclude that
\[
u_\varepsilon>0
\quad\text{in }\Omega .
\]

Thus \(u_\varepsilon\) is a positive weak solution of problem~\eqref{eq:1}.
Moreover, since \(u_\varepsilon\) is a global minimizer of
\(\Phi_\varepsilon\), this solution is a ground state.

It remains to verify the sign of the second derivative of the function
\(t\mapsto\Phi_\varepsilon(tu_\varepsilon)\) at \(t=1\). Put
\[
T_\varepsilon:=\int_\Omega |\nabla u_\varepsilon|^p\,dx,
\qquad
A_\varepsilon:=\int_\Omega a(x)u_\varepsilon^q\,dx,
\qquad
B_\varepsilon:=\int_\Omega b(x)u_\varepsilon^\gamma\,dx .
\]
Since \(u_\varepsilon\) is a critical point, we have
\[
\varepsilon T_\varepsilon-A_\varepsilon+B_\varepsilon=0.
\]
Moreover,
\[
\Phi_\varepsilon(u_\varepsilon)<0 .
\]
Using \(A_\varepsilon=\varepsilon T_\varepsilon+B_\varepsilon\), we obtain
\[
0>
\Phi_\varepsilon(u_\varepsilon)
=
\varepsilon T_\varepsilon\left(\frac1p-\frac1q\right)
+
B_\varepsilon\left(\frac1\gamma-\frac1q\right).
\]
Hence
\[
B_\varepsilon>
\frac{\gamma(q-p)}{p(\gamma-q)}\,\varepsilon T_\varepsilon .
\]
Consequently,
\[
\left.\frac{d^2}{dt^2}\Phi_\varepsilon(tu_\varepsilon)\right|_{t=1}
=
(p-q)\varepsilon T_\varepsilon+(\gamma-q)B_\varepsilon
>
0 .
\]
Assertion \textbf{(2)} is proved.

\medskip

We now prove \textbf{(3)}. Introduce the truncated functional
\[
\Phi_\varepsilon^+(u)
=
\frac{\varepsilon}{p}\int_\Omega |\nabla u|^p\,dx
-
\frac1q\int_\Omega a(x)(u^+)^q\,dx
+
\frac1\gamma\int_\Omega b(x)(u^+)^\gamma\,dx,
\qquad
u^+:=\max\{u,0\}.
\]
The functional \(\Phi_\varepsilon^+\) satisfies the same coercivity estimates
as \(\Phi_\varepsilon\) and satisfies the Palais--Smale condition.

Indeed, let \((u_n)\subset W^{1,p}_0(\Omega)\) be a Palais--Smale sequence
for \(\Phi_\varepsilon^+\), that is,
\[
\Phi_\varepsilon^+(u_n)\ \text{is bounded},
\qquad
D\Phi_\varepsilon^+(u_n)\to0
\quad\text{in }(W^{1,p}_0(\Omega))^* .
\]
By coercivity, \((u_n)\) is bounded in \(W^{1,p}_0(\Omega)\). Hence, passing
to a subsequence, we may assume that
\[
u_n\rightharpoonup u
\quad\text{weakly in }W^{1,p}_0(\Omega),
\]
and
\[
u_n\to u
\quad\text{strongly in }L^r(\Omega),
\qquad
1<r<p^* .
\]
In particular,
\[
u_n\to u
\quad\text{strongly in }L^q(\Omega)
\quad\text{and}\quad
u_n\to u
\quad\text{strongly in }L^\gamma(\Omega).
\]
Since the map \(\mathbb R\ni z\mapsto z^+\) is Lipschitz continuous, we also
have
\[
u_n^+\to u^+
\quad\text{strongly in }L^q(\Omega)
\quad\text{and}\quad
u_n^+\to u^+
\quad\text{strongly in }L^\gamma(\Omega).
\]

Substituting \(u_n-u\) into the derivative \(D\Phi_\varepsilon^+(u_n)\), we
obtain
\[
\begin{aligned}
	o(1)
	&=
	D\Phi_\varepsilon^+(u_n)(u_n-u)
	\\
	&=
	\varepsilon
	\int_\Omega
	|\nabla u_n|^{p-2}\nabla u_n\cdot
	(\nabla u_n-\nabla u)\,dx
	\\
	&\quad
	-
	\int_\Omega
	a(x)(u_n^+)^{q-1}(u_n-u)\,dx
	+
	\int_\Omega
	b(x)(u_n^+)^{\gamma-1}(u_n-u)\,dx .
\end{aligned}
\]
The last two integrals tend to zero. For example,
\[
\left|
\int_\Omega
a(x)(u_n^+)^{q-1}(u_n-u)\,dx
\right|
\le
\|a\|_\infty
\|u_n^+\|_q^{q-1}
\|u_n-u\|_q
\to0,
\]
and, similarly,
\[
\int_\Omega
b(x)(u_n^+)^{\gamma-1}(u_n-u)\,dx
\to0 .
\]
Therefore
\[
\int_\Omega
|\nabla u_n|^{p-2}\nabla u_n\cdot
(\nabla u_n-\nabla u)\,dx
\to0 .
\]
We use the standard \(S^+\)-type property of the \(p\)-Laplacian: if
\[
u_n\rightharpoonup u
\quad\text{weakly in }W^{1,p}_0(\Omega)
\]
and
\[
\limsup_{n\to+\infty}
\int_\Omega
|\nabla u_n|^{p-2}\nabla u_n\cdot
(\nabla u_n-\nabla u)\,dx
\le0,
\]
then
\[
u_n\to u
\quad\text{strongly in }W^{1,p}_0(\Omega).
\]
This property is a standard consequence of the monotonicity of the
\(p\)-Laplacian and the uniform convexity of \(W^{1,p}_0(\Omega)\); see, for
example, the results on monotone operators of type \(S^+\) in
\cite{DrabekMilota}. Applying this property to \((u_n)\), we obtain
\[
u_n\to u
\quad\text{strongly in }W^{1,p}_0(\Omega).
\]
Thus \(\Phi_\varepsilon^+\) satisfies the Palais--Smale condition.

We next verify the mountain-pass geometry. By the Sobolev embedding,
\[
\Phi_\varepsilon^+(u)
\ge
\left(
\frac{\varepsilon}{p}
-
C\|u\|_{W^{1,p}_0}^{q-p}
\right)
\|u\|_{W^{1,p}_0}^{p}.
\]
Choosing \(\rho>0\) sufficiently small, we obtain
\[
\Phi_\varepsilon^+(u)>0
\quad\text{for}\quad
0<\|u\|_{W^{1,p}_0}\le \rho .
\]
In particular, there exists \(\delta>0\) such that
\[
\Phi_\varepsilon^+(u)\ge\delta
\quad\text{whenever}\quad
\|u\|_{W^{1,p}_0}=\rho .
\]
On the other hand, the ground state \(u_\varepsilon\) found above is
nonnegative and satisfies
\[
\Phi_\varepsilon^+(u_\varepsilon)
=
\Phi_\varepsilon(u_\varepsilon)
<0 .
\]
Therefore \(\|u_\varepsilon\|_{W^{1,p}_0}>\rho\), and
\(\Phi_\varepsilon^+\) has the mountain-pass geometry.

By the mountain pass theorem \cite{AmbrRabin}, there exists a critical point
\(v_\varepsilon\in W^{1,p}_0(\Omega)\) such that
\[
\Phi_\varepsilon^+(v_\varepsilon)
=
c_\varepsilon
:=
\inf_{\xi\in\Gamma}
\max_{t\in[0,1]}
\Phi_\varepsilon^+(\xi(t))
\ge
\delta>0,
\]
where
\[
\Gamma
=
\left\{
\xi\in C([0,1];W^{1,p}_0(\Omega)):
\xi(0)=0,\ \xi(1)=u_\varepsilon
\right\}.
\]
In particular,
\[
v_\varepsilon\not\equiv0 .
\]

We show that \(v_\varepsilon\ge0\). Since
\[
D\Phi_\varepsilon^+(v_\varepsilon)=0,
\]
we take \(-v_\varepsilon^-\) as a test function, where
\[
v_\varepsilon^-:=\max\{-v_\varepsilon,0\}.
\]
The nonlinear terms involving \(v_\varepsilon^+\) vanish, and we obtain
\[
\varepsilon
\int_\Omega |\nabla v_\varepsilon^-|^p\,dx
=
0 .
\]
Consequently,
\[
v_\varepsilon^-=0,
\qquad
v_\varepsilon\ge0
\quad\text{in }\Omega .
\]
Thus \(v_\varepsilon^+=v_\varepsilon\), and \(v_\varepsilon\) is a critical
point of the original functional \(\Phi_\varepsilon\). Hence
\(v_\varepsilon\) is a weak solution of problem~\eqref{eq:1}.

Since
\[
\Phi_\varepsilon(v_\varepsilon)
=
\Phi_\varepsilon^+(v_\varepsilon)
=
c_\varepsilon
>
0,
\]
this solution is distinct from the ground state \(u_\varepsilon\), because
\[
\Phi_\varepsilon(u_\varepsilon)<0 .
\]

As above, the a priori estimates and interior regularity results for the
\(p\)-Laplacian imply that
\[
v_\varepsilon\in C^{1,\kappa}_{\mathrm{loc}}(\Omega)
\]
for some \(\kappa\in(0,1)\). Finally, since
\(v_\varepsilon\ge0\), \(v_\varepsilon\not\equiv0\), and
\[
-\varepsilon\Delta_p v_\varepsilon
+
C v_\varepsilon^{p-1}
\ge0
\quad\text{in the weak sense in }\Omega
\]
with some constant \(C>0\), the strong maximum principle for the
\(p\)-Laplacian (see, for instance, \cite{Vazquez1984}) gives
\[
v_\varepsilon>0
\quad\text{in }\Omega .
\]
Assertion \textbf{(3)} is proved.


\section{Proof of Theorem~\ref{thmCC2}}

Set
\[
J(u):=
-\frac1q\int_\Omega a(x)|u|^q\,dx
+
\frac1\gamma\int_\Omega b(x)|u|^\gamma\,dx .
\]
For \(s\ge0\), denote
\[
j_x(s):=
-\frac{a(x)}q s^q+\frac{b(x)}\gamma s^\gamma .
\]
Then
\[
J(u)=\int_\Omega j_x(|u(x)|)\,dx .
\]

Since \(a(x)\ge\sigma_a>0\) and \(b(x)\ge\sigma_b>0\) for a.e.
\(x\in\Omega\), the function \(s\mapsto j_x(s)\) has a unique positive
global minimizer on \([0,+\infty)\), namely
\[
\bar u_0(x)
=
\left(\frac{a(x)}{b(x)}\right)^{1/(\gamma-q)} .
\]
Hence
\begin{equation}\label{eq:J-min}
	J(u)\ge J(\bar u_0)
\end{equation}
for every \(u\in L^\gamma(\Omega)\). Moreover, the assumptions
\(a,b\in L^\infty(\Omega)\), \(a\ge\sigma_a>0\), and
\(b\ge\sigma_b>0\) imply that
\[
\bar u_0\in L^\infty(\Omega),
\qquad
\bar u_0>0
\quad\text{for a.e. }x\in\Omega .
\]

Let \(u_\varepsilon\) be a positive ground state of problem~\eqref{eq:1},
\(0<\varepsilon<\varepsilon_e^*\). We first note that
\(u_\varepsilon\) is a global minimizer of \(\Phi_\varepsilon\) on
\(W^{1,p}_0(\Omega)\). Indeed, by the proof of Theorem~\ref{thm1}, the
functional \(\Phi_\varepsilon\) has a global minimizer \(z_\varepsilon\),
and this minimizer is a weak solution of \eqref{eq:1}. Since
\(u_\varepsilon\) is a ground state, we have
\[
\Phi_\varepsilon(u_\varepsilon)
\le
\Phi_\varepsilon(z_\varepsilon).
\]
On the other hand, by the global minimality of \(z_\varepsilon\),
\[
\Phi_\varepsilon(z_\varepsilon)
\le
\Phi_\varepsilon(u_\varepsilon).
\]
Thus
\[
\Phi_\varepsilon(u_\varepsilon)
=
\inf_{u\in W^{1,p}_0(\Omega)}\Phi_\varepsilon(u).
\]
Therefore, from \eqref{eq:J-min}, we obtain
\begin{equation}\label{eq:lower-energy}
	\Phi_\varepsilon(u_\varepsilon)
	\ge
	J(u_\varepsilon)
	\ge
	J(\bar u_0).
\end{equation}

On the other hand, since \(C_0^\infty(\Omega)\) is dense in
\(L^\gamma(\Omega)\), for every \(\delta>0\) one can choose a nonnegative
function \(w_\delta\in C_0^\infty(\Omega)\) such that
\[
\|w_\delta-\bar u_0\|_\gamma<\delta .
\]
Then, by the global minimality of \(u_\varepsilon\),
\[
\Phi_\varepsilon(u_\varepsilon)
\le
\Phi_\varepsilon(w_\delta)
=
\frac{\varepsilon}{p}\int_\Omega |\nabla w_\delta|^p\,dx
+
J(w_\delta).
\]
Consequently,
\[
\limsup_{\varepsilon\to0^+}
\Phi_\varepsilon(u_\varepsilon)
\le
J(w_\delta).
\]
Since \(q<\gamma\) and \(\Omega\) is bounded, the functional \(J\) is
continuous with respect to strong convergence in \(L^\gamma(\Omega)\).
Passing to the limit as \(\delta\to0\), we get
\begin{equation}\label{eq:upper-energy}
	\limsup_{\varepsilon\to0^+}
	\Phi_\varepsilon(u_\varepsilon)
	\le
	J(\bar u_0).
\end{equation}
It follows from \eqref{eq:lower-energy} and \eqref{eq:upper-energy} that
\begin{equation}\label{eq:energy-limit}
	\Phi_\varepsilon(u_\varepsilon)
	\to
	J(\bar u_0)
	\quad\text{as } \varepsilon\to0^+ .
\end{equation}
Moreover,
\[
0
\le
J(u_\varepsilon)-J(\bar u_0)
\le
\Phi_\varepsilon(u_\varepsilon)-J(\bar u_0),
\]
and therefore
\begin{equation}\label{eq:J-limit}
	J(u_\varepsilon)
	\to
	J(\bar u_0)
	\quad\text{as } \varepsilon\to0^+ .
\end{equation}

We now prove that \(u_\varepsilon\to\bar u_0\) in measure. First observe
that the family \(\{u_\varepsilon\}\) is bounded in \(L^\gamma(\Omega)\) as
\(\varepsilon\to0^+\). Indeed, since
\(a,b\in L^\infty(\Omega)\), \(b\ge\sigma_b>0\), and \(q<\gamma\), there
exist constants \(c_1,c_2>0\), independent of \(x\) and \(s\ge0\), such that
\[
j_x(s)
=
-\frac{a(x)}q s^q+\frac{b(x)}\gamma s^\gamma
\ge
c_1s^\gamma-c_2 .
\]
This estimate and \eqref{eq:J-limit} imply that
\(\{u_\varepsilon\}\) is bounded in \(L^\gamma(\Omega)\) for all sufficiently
small \(\varepsilon>0\).

Next, since \(u_\varepsilon\ge0\), from \eqref{eq:J-limit} we have
\begin{equation}\label{eq:integral-gap}
	\int_\Omega
	\bigl(
	j_x(u_\varepsilon(x))-j_x(\bar u_0(x))
	\bigr)\,dx
	\to0 .
\end{equation}
The integrand is nonnegative by the definition of \(\bar u_0(x)\).

Put
\[
\hat a:=\|a\|_\infty,
\qquad
\hat b:=\|b\|_\infty .
\]
Then, for a.e. \(x\in\Omega\),
\[
0<\rho_-
:=
\left(\frac{\sigma_a}{\hat b}\right)^{1/(\gamma-q)}
\le
\bar u_0(x)
\le
\left(\frac{\hat a}{\sigma_b}\right)^{1/(\gamma-q)}
=:\rho_+
<+\infty .
\]

Fix \(\eta>0\). We claim that there exists \(\kappa_\eta>0\) such that
\begin{equation}\label{eq:uniform-separation}
	j_x(s)-j_x(\bar u_0(x))
	\ge
	\kappa_\eta
\end{equation}
for all \(s\ge0\) satisfying
\[
|s-\bar u_0(x)|\ge\eta,
\]
and for a.e. \(x\in\Omega\).

Consider the family of functions
\[
j_{\alpha,\beta}(s)
:=
-\frac{\alpha}{q}s^q+\frac{\beta}{\gamma}s^\gamma,
\qquad
(\alpha,\beta)\in[\sigma_a,\hat a]\times[\sigma_b,\hat b].
\]
For every pair \((\alpha,\beta)\), the function \(j_{\alpha,\beta}\) has a
unique positive global minimizer
\[
\rho(\alpha,\beta)
=
\left(\frac{\alpha}{\beta}\right)^{1/(\gamma-q)} ,
\]
and
\[
\rho(\alpha,\beta)\in[\rho_-,\rho_+].
\]

First observe that
\[
j_{\alpha,\beta}(s)-j_{\alpha,\beta}(\rho(\alpha,\beta))
\to+\infty
\quad\text{as }s\to+\infty
\]
uniformly with respect to
\[
(\alpha,\beta)\in[\sigma_a,\hat a]\times[\sigma_b,\hat b].
\]
Indeed,
\[
j_{\alpha,\beta}(s)
\ge
-\frac{\hat a}{q}s^q+\frac{\sigma_b}{\gamma}s^\gamma ,
\]
whereas the function
\[
(\alpha,\beta)
\mapsto
j_{\alpha,\beta}(\rho(\alpha,\beta))
\]
is continuous on the compact set
\[
[\sigma_a,\hat a]\times[\sigma_b,\hat b].
\]
Thus there exists \(M_0>0\) such that
\[
\bigl|j_{\alpha,\beta}(\rho(\alpha,\beta))\bigr|
\le
M_0
\]
for all \((\alpha,\beta)\in[\sigma_a,\hat a]\times[\sigma_b,\hat b]\).
Hence
\[
j_{\alpha,\beta}(s)-j_{\alpha,\beta}(\rho(\alpha,\beta))
\ge
\frac{\sigma_b}{\gamma}s^\gamma
-
\frac{\hat a}{q}s^q
-
M_0 .
\]
Since \(\gamma>q\), the right-hand side tends to \(+\infty\) as
\(s\to+\infty\), uniformly with respect to \((\alpha,\beta)\). Therefore,
one can choose \(M>0\) such that
\begin{equation}\label{eq:large-s-separation}
	j_{\alpha,\beta}(s)-j_{\alpha,\beta}(\rho(\alpha,\beta))
	\ge1
\end{equation}
for all \(s\ge M\) and all
\[
(\alpha,\beta)\in[\sigma_a,\hat a]\times[\sigma_b,\hat b].
\]

It remains to consider the region \(0\le s\le M\). Let
\[
K_\eta
:=
\left\{
(\alpha,\beta,s):
\alpha\in[\sigma_a,\hat a],\
\beta\in[\sigma_b,\hat b],\
s\in[0,M],\
|s-\rho(\alpha,\beta)|\ge\eta
\right\}.
\]
If \(K_\eta\ne\varnothing\), then \(K_\eta\) is compact. The function
\[
(\alpha,\beta,s)
\mapsto
j_{\alpha,\beta}(s)-j_{\alpha,\beta}(\rho(\alpha,\beta))
\]
is continuous and strictly positive on \(K_\eta\), since equality to zero is
possible only when \(s=\rho(\alpha,\beta)\), and such points are excluded by
the condition \(|s-\rho(\alpha,\beta)|\ge\eta\). Therefore
\[
\kappa_\eta'
:=
\min_{K_\eta}
\bigl[
j_{\alpha,\beta}(s)-j_{\alpha,\beta}(\rho(\alpha,\beta))
\bigr]
>0 .
\]
If \(K_\eta=\varnothing\), set, for instance, \(\kappa_\eta'=1\). In both
cases, define
\[
\kappa_\eta:=\min\{\kappa_\eta',1\}>0.
\]
Then \eqref{eq:large-s-separation} and the definition of \(\kappa_\eta'\)
imply that
\[
j_{\alpha,\beta}(s)-j_{\alpha,\beta}(\rho(\alpha,\beta))
\ge
\kappa_\eta
\]
for all \(s\ge0\) and all
\((\alpha,\beta)\in[\sigma_a,\hat a]\times[\sigma_b,\hat b]\) such that
\[
|s-\rho(\alpha,\beta)|\ge\eta .
\]
Applying this assertion with
\[
\alpha=a(x),
\qquad
\beta=b(x),
\qquad
\rho(\alpha,\beta)=\bar u_0(x),
\]
we obtain \eqref{eq:uniform-separation}.

Consequently,
\[
\kappa_\eta
\operatorname{meas}
\{x\in\Omega:\ |u_\varepsilon(x)-\bar u_0(x)|\ge\eta\}
\le
\int_\Omega
\bigl(
j_x(u_\varepsilon(x))-j_x(\bar u_0(x))
\bigr)\,dx .
\]
By \eqref{eq:integral-gap}, the right-hand side tends to zero. Hence, for
every \(\eta>0\),
\[
\operatorname{meas}
\{x\in\Omega:\ |u_\varepsilon(x)-\bar u_0(x)|\ge\eta\}
\to0,
\]
that is,
\begin{equation}\label{eq:measure-convergence}
	u_\varepsilon\to\bar u_0
	\quad\text{in measure in }\Omega .
\end{equation}

We now prove strong convergence in \(L^r(\Omega)\) for \(1\le r<\gamma\).
Since \(\{u_\varepsilon\}\) is bounded in \(L^\gamma(\Omega)\) and
\(\bar u_0\in L^\gamma(\Omega)\), the family
\(\{u_\varepsilon-\bar u_0\}\) is bounded in \(L^\gamma(\Omega)\).
Therefore, for every \(1\le r<\gamma\), the family
\(\{|u_\varepsilon-\bar u_0|^r\}\) is bounded in
\(L^{\gamma/r}(\Omega)\), where \(\gamma/r>1\). Hence it is uniformly
integrable. Together with \eqref{eq:measure-convergence}, Vitali's
convergence theorem gives
\[
\int_\Omega |u_\varepsilon-\bar u_0|^r\,dx\to0 .
\]
Thus
\[
u_\varepsilon\to\bar u_0
\quad\text{strongly in }L^r(\Omega),
\qquad
1\le r<\gamma .
\]

Moreover,
\[
u_\varepsilon\rightharpoonup\bar u_0
\quad\text{weakly in }L^\gamma(\Omega).
\]
Indeed, the family \(\{u_\varepsilon\}\) is bounded in the reflexive space
\(L^\gamma(\Omega)\). Hence every sequence \(\varepsilon_m\to0^+\) contains
a subsequence converging weakly in \(L^\gamma(\Omega)\). By
\eqref{eq:measure-convergence}, any such weak limit must coincide with
\(\bar u_0\). Therefore the whole family \(u_\varepsilon\) converges weakly
to \(\bar u_0\) in \(L^\gamma(\Omega)\). Since \(\Omega\) is bounded, it
follows in particular that
\[
u_\varepsilon-\bar u_0\rightharpoonup0
\quad\text{weakly in }L^r(\Omega),
\qquad
1<r\le\gamma .
\]

Finally, by the definition of \(\bar u_0\), we have
\[
a(x)\bar u_0^{q-1}(x)
-
b(x)\bar u_0^{\gamma-1}(x)
=
0
\quad\text{for a.e. }x\in\Omega .
\]
Thus \(\bar u_0\) satisfies the limiting equation \eqref{eq:1} with
\(\varepsilon=0\):
\[
a(x)\bar u_0^{q-1}
=
b(x)\bar u_0^{\gamma-1}
\quad\text{a.e. in }\Omega .
\]
The theorem is proved.

\bibliographystyle{elsarticle-num}

\end{document}